\documentclass[reqno]{amsart}

\usepackage{amsmath, amssymb, amsthm}
\usepackage{cite, hyperref}
\usepackage{epsfig}
\usepackage{graphicx}
\usepackage{caption}

\def\tr{\operatorname{Tr}}
\def\ev{\operatorname{ev}}
\def\size#1{\mathbf S\left(#1 \right)}
\def\poly{\operatorname{\emph{Poly}}}
\let\polishL=l

\newcommand{\Z}{\mathbb{Z}}
\newcommand{\Q}{\mathbb{Q}}
\newcommand{\R}{\mathbb{R}}
\newcommand{\C}{\mathbb{C}}

\renewcommand{\P}{\mathbb{P}}

\newtheorem{proposition}{Proposition}
\newtheorem{corollary}{Corollary}
\newtheorem{theorem}{Theorem}
\newtheorem{lemma}{Lemma}

\theoremstyle{definition}
\newtheorem{definition}{Definition}

\providecommand{\norm}[1]{\lVert#1\rVert}

\numberwithin{equation}{section}
\numberwithin{proposition}{section}
\numberwithin{corollary}{section}

\begin{document}
\title
 { On the number of zeros of Melnikov functions }

\author
 {Sergey Benditkis,
  Dmitry Novikov
 }

\address{Weizmann Institute of Science\\Rehovot\\Israel}
\email{\tt sergey.benditkis\@weizmann.ac.il, dmitry.novikov\@weizmann.ac.il}
\thanks{This research was supported by  Israel Science Foundation (grant No. 1501/08).}
\subjclass[2000]{Primary 34C07, 34C08; Secondary  37F75, 32S65}

\begin{abstract}
    We provide an effective uniform upper bond for the number of zeros of the first non-vanishing Melnikov function of a polynomial perturbations of a planar polynomial Hamiltonian vector field. The bound depends on degrees of the field and of the perturbation, and on the order $k$ of the Melnikov function. The  generic case $k=1$ was considered by  Binyamini, Novikov and Yakovenko (\cite{BNY-Inf16}).  The bound follows from an effective construction of the Gauss-Manin connection for iterated integrals.
\end{abstract}
\setcounter{tocdepth}{2}

\maketitle

%\tableofcontents
%\pagebreak
\section{Introduction}
\subsection{Infinitesimal Hilbert 16th problem}
The second part of 16th Hilbert problem asks: How many limit cycles may have a planar polynomial vector field? The question has a long history, and was at the origin of several theories, see \cite{centennial}). The most remarkable achievement, Ecalle-Ilyashenko theorem, claims that the number of limit cycles is finite for any individual vector field, see \cite{ecalle:book,ilyashenko:finiteness}. However, existence of a uniform upper bound for this number  even for quadratic vector fields is an open problem.

A weaker form of the same question concerns perturbations of Hamiltonian vector fields.
Let $H(x,y)$ be a bivariate polynomial (further called Hamiltonian). The corresponding Hamiltonian system
can be written in Pfaffian form as
\begin{equation}\label{unperturbed}
    dH = 0.
\end{equation}
Consider its polynomial perturbation
\begin{equation}\label{perturbed}
    dH+\varepsilon\omega = 0,\quad \text{where}\quad  \omega = P(x, y)dx + Q(x, y)dy, \quad P,Q\in \R[x,y],
\end{equation}
and $\varepsilon \in (\R^1, 0)$.
%We will always assume that $\deg\omega<\deg H$. This restriction does not seem to be crucial, as  general technology  of Petrov as developed by Khovanskii
%allow to extend this to forms of any degree, see \cite{Zol}.

Consider a nest of cycles $\left\{\delta_t\subset\{H=t\}, t\in[a,b]\subset\R\right\}$ of (\ref{unperturbed}). We ask how many limit cycles of (\ref{perturbed}) converge to this nest as $\varepsilon\to 0$.

It is easy to see that  closed trajectories $\delta_t$ that survive after the perturbation should produce zero value of  the Poincar\'e integral (aka first Melnikov function)
$$
    I = I(\delta_t, \omega) = \oint_{\delta_t}\omega,
$$
the so called Poincar\'e-Andronov-Pontryagin criterion, see \cite[\S 26A]{iy:lade}. Therefore estimates on the number of zeros of this so-called Abelian integral have direct relation to the Hilbert 16th problem.  Binyamini, Novikov, Yakovenko studied the case of \emph{non-conservative} perturbations, namely, when the Poincar\'e integral does not vanishes identically.
\begin{theorem}\cite{BNY-Inf16}\label{thm:BNY}
Assume that $I\not\equiv 0$ for the nest of cycles of (\ref{unperturbed}). Assume that $\deg\omega<\deg H$. Then the number of cycles $\delta_t$ providing the zero value of Poincar\'e integral is at most $2^{2^{P(\deg H)}}$, where $P(n)$ is some explicit polynomial of degree at most 60.
\end{theorem}
This upper bound serves also as an upper bound for the cyclicity of an open nest of the limit cycles (which is defined as a supremum of cyclicities of all closed subnests of the open nest, see e.g. \cite{novikov-gavrilov-2008}).

For generic Hamiltonians  identical vanishing of $I$ implies exactness of $\omega$ (again, assuming $\deg\omega<\deg H$), so the perturbation remains integrable, see \cite{I69}. However, for degenerate Hamiltonians one has to consider Melnikov functions of higher order.

\subsection{Melnikov functions and the main theorem}
\begin{definition}\label{def:Melnikov}
For a cycle $\delta$ of (\ref{unperturbed}) choose a transversal $\sigma$ with coordinate $z$ chosen in such a way that $\delta$ intersect $\sigma$ at $z=0$.

Denote by $ \Delta:\sigma\to\sigma$ the \emph{holonomy map} of cycle $\gamma$ considered as a function of the parameters $h, \varepsilon$.
Being analytic function of its arguments,  $\Delta$ can be expanded in the converging series
\begin{equation}\label{eq:M_k def}
   \Delta(z,\varepsilon) = z+\varepsilon M_1(z)+\dots+\varepsilon^K M_K(z)+\dots,
\end{equation}
where $M_k(z)$ are real analytic functions defined in some common neighborhood of the origin $z=0$.
The function $M_k$ is called $k$-th Melnikov function.
\end{definition}

Assume that the first nonzero function $M_k(z)$  has $N$ isolated zeroes (counted with their multiplicities) in the closed interval $\{|z|\leq\rho\}$.
\begin{proposition}\label{zeroes-i-k}\cite[Proposition 26.1]{iy:lade}
There exists a small positive value $r>0$ such that for all $|\varepsilon|<r$ the foliation \eqref{perturbed}  has no more than $n$ limit cycles intersecting $\sigma$ at $\{|h|\leq\rho\}$.
\end{proposition}

Our main result provides an upper bound for the number of isolated zeros of the  first non-zero Melnikov function.
\begin{theorem}\label{num-of-zeroes-m-k} The number of isolated zeroes of the first non-zero Melnikov function $M_K$ is bounded by $2^{2^{d^{O(1)}n^{O(K)}}}$, where $n+1=\deg H$, $d=\deg\omega$, and the absolute constants in $O(1),O(K)$ can be explicitly computed.
\end{theorem}
This bound is certainly not exact, and construction of lower bounds is a difficult problem, unsolved even for Abelian integrals.

Note that the order $K$ of the  first non-zero Melnikov function cannot be easily bounded in terms of degree of $H$: this problem includes, as a particular case, the center-focus problem.

\subsection{Iterated integrals and algebraic motivation}

It  is well-known, see \cite{gavrilov:iterated, iy:lade}, that $M_K$ can be represented as a linear combination of so-called \emph{iterated integrals} of order at most $K$.

\begin{definition}
Let $\gamma(s):[0,1]\to\C^2$ be parameterization of a curve $\gamma\subset \C^2$.  For a $k$-tuple of forms $\omega_1,...,\omega_k\in\Lambda^1(\C^2)$ we define the iterated integral as
$$
\int_\gamma\omega_1...\omega_k=\int_0^1\left(\int_0^{s_1}\left(...\left(\int_0^{s_k}\gamma^*\omega_k\right)\gamma^*\omega_{k-1}\right)...\right)\gamma^*\omega_1.
$$
\end{definition}
Iterated integrals were extensively studied from various points of view, see e.g. \cite{chen, harris, movasati}. Our goal is to investigate their oscillation properties.
Let us choose a straight line as a transversal to the nest of cycles.
Iterated integrals define functions on this transversal: to any point $p$ of the transversal corresponds the value of the iterated integral over the cycle of the foliation passing through it, with $p$ being the initial point of the path of integration (note that, unlike the Melnikov function, the iterated integrals do depend on the choice of the initial point of the cycle).

The main step of the proof of Theorem~\ref{num-of-zeroes-m-k} is an explicit construction of a meromorphic flat connection  whose horizontal sections are given by basic iterated integrals (see \eqref{Sk-def} for definition), a higher order analogue of the Gauss-Manin connection for Abelian integrals. We prove in Section~\ref{sec:properties} that this connection belongs to the class of connections considered in the paper of of Binyamini, Novikov and Yakovenko~\cite{BNY-Inf16}, see the next section for formulation of the result. Estimates on the complexity of the connection, proved in  Section~\ref{sec:construction}, allow to apply their main result not only to linear combinations of basic iterated integrals, but also to their combinations with coefficients polynomially  dependent on $z$ from \eqref{eq:M_k def}. In Section~\ref{sec:Mk as II} we represent $M_K$ in this form.

%%%%%%%%%%%%%%%%%%%%%%%%%%%%%

\section{Non-oscillation of horizontal sections of meromorphic connections}\label{sec:BNYresult}
In this section we briefly recall the main result of \cite{BNY-Inf16}.
Let  $\Omega$ be a rational $l\times l$-matrix of rational differential 1-forms on a complex manifold $M$, with a singular locus $\Sigma $. It defines a connection
\begin{equation}\label{picard-fuchs}
    dX = \Omega \cdot X
\end{equation}
 on trivial vector bundle $M\times \C^l$. We denote by $\Sigma$ the singular locus of the connection.
\subsection{Regular integrable connections}
\begin{definition}
The form $\Omega$ is \emph{integrable} or \emph{locally flat} if $d\Omega - \Omega \wedge \Omega = 0$.
\end{definition}
This condition is equivalent to local existence of a basis of horizontal sections of  \eqref{picard-fuchs}
near each nonsingular point $a\notin \Sigma$.

\begin{definition}
   The Picard-Fuchs system \eqref{picard-fuchs} (and the corresponding matrix 1-form $\Omega$) is called \emph{regular}
at  $a\in M$, if for any germ of a holomorpic curve $\gamma : (\C, 0) \to (M, a)$ the pull-back of the connection to $(\C,0)$ has a regular singularity at the origin:
\begin{equation}
                \forall C>0  \quad\exists p=p(C)\in \R \quad  \|X(\gamma(s))\|^{\pm 1} =O(|s|^{-p})
\end{equation}
as $s\to 0$ in the  sector $\{\arg s|\le C\}$.

Connection is called regular on $M$ if it is regular at each point $a\in M$.
\end{definition}
Regular connections remain regular after pull-backs, push-forwards, (semi)direct products etc., see \cite{deligne}.

\subsection{Quasiunipotent connections}
\begin{definition} For a point $a\in M$ a \emph{small loop around $a$} is a closed path $\gamma$, such that exists  a mapping $\{|z|\le 1\} \to M$ which maps $0$ to $a$,  $\{|z|=1\}$ to $\gamma$  and such that
 the image of $\{|z|\le 1\} \setminus \{0\}$ is
disjoint with $\Sigma$.
\end{definition}

Recall that an operator is called quasiunipotent if all its eigenvalues are roots of unity, i.e. belong to $\exp(2\pi i \Q)$.
\begin{definition}
The integrable form $\Omega$ is called \emph{quasiunipotent} at a point $a \in M$, if monodromy operator associated with any small loop
around $a$ is quasiunipotent. The system is (globally) quasiunipotent, if it is quasiunipotent at
every point of $\C P^n$.
\end{definition}
In general, this does not mean that \emph{every} monodromy operator associated to $\Omega$ is quasiunipotent.
\begin{theorem}
(Kashiwara theorem \cite{kashiwara:qu}). A regular integrable system that
is quasiunipotent at each point outside an algebraic subset of codimension 2, is globally quasiunipotent.
\end{theorem}

\subsection{Degree of rational function} We define degree of a rational function to be the minimum of sums of degrees of numerator and denominator over all its
representations as a ratio of two polynomials. Degree of the form is defined in such a way that the operator $d$ has degree $0$.
  
\subsection{Notion of size} In this work, similar to \cite{BNY-Inf16}, we are studying various objects, like matrices, functions, differential forms, defined over $\Q$ - field of rational numbers. To obtain quantitative characteristics of these objects, we need to use the notion of \emph{size}, or \emph{complexity} of the objects.

\begin{definition}
The norm of a multivariate polynomial $P \in \C[z_1, \dots, z_n]$,
$
P(z) =\sum_\alpha c_\alpha z^\alpha$ (in the standard multiindex notation) is the sum
of absolute values of its coefficients, $\norm{P} = \sum_\alpha|c_\alpha|$. Clearly, this norm is
multiplicative,
$$
\norm{PQ} \leq \norm{P} \cdot \norm{Q}
$$
\end{definition}
\begin{definition}
The size $\textbf{S}(P)$ of an integer polynomial $P \in \Z[z_1, \dots, z_n]$ is
set to be equal to its norm, $\textbf{S}(P) = \norm{P}$.
\par The size of a rational fraction $R \in \Q(z_1, \dots , z_n)$ is
$$
\textbf{S}(R) = \min_{P,Q}\{\norm{P} + \norm{Q}: R = P/Q; P,Q \in \Z[z_1, \dots, z_n]\}
$$
\par The size of a (polynomial or rational) 1-form on $\P^m$ or on $\P^m\times \P^1$ defined over $\Q$, is the
sum of sizes of its coefficients in the standard affine chart $\C^m$.
\par The size of a vector or matrix rational function (resp., 1-form) defined
over $\Q$, is the sum of the sizes of its components.
\end{definition}

Note that, unlike polynomials, for rational functions we have only
\begin{equation}\label{eq:deg sum rat}
\begin{split}
\deg\left(\sum R_i\right)\le 2\sum_i\deg R_i\\
\textbf{S}\left(\sum_{i=1}^n R_i\right)\le (n+1)\prod_{i=1}^n \textbf{S}(R_i).
\end{split}
\end{equation}

\subsubsection{Counting  number of zeroes of the solution}
Let $\Omega$ be a rational $l \times l$-matrix  1-form of degree $d$ on the product $\C P^m$, and consider the restriction of the corresponding Picard-Fuchs system \eqref{picard-fuchs} to some line $\ell\cong\C P^1\subset\C P^m$. We are interested in the number of zeros of a linear combinations of entries of the fundamental matrix of \eqref{picard-fuchs}. In general, restriction of the fundamental matrix to this line produces a multivalued matrix function on $\ell\setminus\Sigma$, so to count zeros one should choose a simply connected domain in $\ell\setminus\Sigma$. One can easily see that the geometric complexity of the domain should be taken into account.
\begin{definition}
We denote by $\mathcal{N}(\ell) = \mathcal{N}(\Omega|_\ell)$ the supremum over all constant matrices $B$ and all triangles $T$ lying in $\ell\setminus\Sigma$ of the number of isolated zeroes of the function $\tr BX$ in $T$.
\end{definition}

\begin{theorem}\label{th:upper-bound} \cite[Theorems 7,8]{BNY-Inf16}
Let $\Omega$ be a rational $l \times l$-matrix  1-form of degree $d$ on the product $\C P^m\times \C P^1$. Assume that $\Omega$ is integrable, regular and quasiunipotent, and its size is $s = \textbf{S}(\Omega)$. Then
$$
    \forall \ell\cong \C P^1 \subset \C P^m\times \C P^1 ~~~~ \mathcal{N}(\ell)\leq s^{2^{C(dl^4m)^5}}
$$
for some universal constant $C$.
\end{theorem}

\section{Construction of Gauss-Manin connection for iterated integrals}\label{sec:construction}

\subsection{Base spaces: notations}\label{ssec:def bases}
Let $\C_{n+1}[x,y]$ denote the space of all bivariate polynomials of degree at most $n+1$. We will denote the points of its projectivisation 
$P\C_{n+1}[x,y]$ by $\lambda$. In standard coordinates $H=\sum_{0\le i+j\le n+1}\lambda_{ij}x^iy^j$. By $\tilde{\lambda}$ we denote the tuple of all elements of $\lambda$ except the last one, $\lambda_{00}$.

An important role plays the space $\widetilde{\C}_{n+1}[x,y]$ of polynomials vanishing at the origin, of dimension smaller by 1. The tuples $\tilde{\lambda}$ parameterize the points of its projectivisation $P\widetilde{\C}_{n+1}[x,y]$.

\subsection{Gauss-Manin connection for Abelian integrals}

\begin{definition}
    Let $H\in \C[x,y]$ be a polynomial of degree $n+1$.
     \emph{Petrov module} $P_H$ is a $\C[t]$-module defined as quotient space
     $$
        \textbf{P}_H = \frac{\Lambda^1}{dH\cdot\Lambda^0 + d\Lambda^0}
     $$
     of polynomial 1-forms over a space of relatively exact forms $f\cdot dH+dg$, where $f,g$ are polynomials.
\end{definition}

\begin{proposition}\label{prop:petrov-AI}\cite[Theorem 26.21]{iy:lade}
The set of all Morse-plus polynomials $H$ for which the forms $\omega_{ij}=x^{i-1}y^j\,dx$, $1\le i,j\le n$ form a basis of $P_H$ over $\C[h]$ is a Zarisky open subset in $P\C_{n+1}[x,y]$.
\end{proposition}

Forms $\omega_{ij}$ provide a convenient trivialization of homological Milnor bundle over $P\C_{n+1}[x,y]$. The  Gauss-Manin connection in this trivialization can be written explicitly and links the main result of \cite{BNY-Inf16} to Infinitesimal Hilbert 16th problem. Let us formulate this result.

Let $H$ be a polynomial satisfying conditions of Proposition~\ref{prop:petrov-AI}, such that the  affine curve $\Gamma_H=\{H=0\}\subset \C^2$ is smooth. Choose a point $p_0\in\Gamma_H$. $\Gamma_H$  is a Riemann surface of genus $\frac{n(n-1)}2$ with $n+1$ removed points. Therefore its fundamental group $\pi_1(\Gamma_H, p_0)$ is a free group in $N=n^2$ generators.

Choose $\delta_1, ... \delta_{N}\in F$  in such a way that their homology classes form a basis in $H_1(\Gamma_H, \Z)$.
\begin{theorem}[\cite{BNY-Inf16}]\label{th:non-deg-AI}
The matrix
\begin{equation}\label{eq:S1def}
{S}_1=\left\{\oint_{\delta_k}\omega_l\right\}_{k,l=1}^N,
\end{equation}
 where $\omega_l$ is an enumeration of the set  of basic forms $\mathcal{B}=\{\omega_{ij}=x^{i-1}y^j\,dx, i,j=1,...,n\}$, is non-degenerate.
Moreover, ${S}_1={S}_1(H)$ is the matrix of fundamental solutions of the Picard-Fuchs equation
\begin{equation}\label{eq-pfs}
d{S}_1=\Omega_1{S}_1,
\end{equation}
which is defined over $\Q$  and has the size $s=\size{\Omega}$, dimension $\ell$ and the degree $d=\deg
\Omega$  explicitly
bounded from above as
\begin{equation}\label{bounds-pfs}
    s\le 2^{\poly(n)},\quad d\le O(n^2),\quad \ell= n^2.
\end{equation}

\end{theorem}

Using these estimates and Theorem~\ref{th:upper-bound}, one get the main result of \cite{BNY-Inf16}.

Our goal is to generalize this construction for iterated integrals of length $K>1$. To this end we will need more detailed results.

\begin{proposition}\label{prop:petrov}
Let $\theta$ be a polynomial one-form of degree $d$ on $\C^2_{x,y}$, and assume that it  it is defined over $\Q(\lambda)$, 
is of degree $\nu$ in $\lambda$ and of size $s$.
Denote by $\tilde{\lambda}=\lambda\setminus\{\lambda_{00}\}$ the tuple of all coefficients of $H$ except the first one. Then one can write a decomposition
\begin{equation}\label{eq:petrov decomp}
\theta=\sum_{i=1}^N \left(f_i\circ H\right) \omega_i+fdH+dg, \qquad f,g\in\Q(\tilde{\lambda})[x,y], \quad f_i\in\Q(\tilde{\lambda})[h],
\end{equation}
with $\deg_{x,y} f, \deg_{x,y} g \le d$ and $\deg_h f_i\le \frac{d}{n+1}$. Moreover, coefficients of $f,g$ and $f_i$ are of degree at most $\nu+O(d^3)$ in $\tilde{\lambda}$ and their sizes are bounded  by $sd^{O(d^3)}$.
\end{proposition}
\begin{proof}

It is well known that for any fixed sufficiently generic $\tilde{\lambda}$ one can write decomposition \eqref{eq:petrov decomp} with this bounds on degrees in $x,y$ and $h$ and some numerical coefficients, see e.g. \cite{G98}. To understand dependence on $\tilde{\lambda}$, consider \eqref{eq:petrov decomp} as a system of linear equations on the coefficients of $f, g$ and $f_i$. Assume that $d>n$. The number of equations (i.e. of coefficients of $\theta$) is  $(d+1)(d+2)=O(d^2)$. The number of unknowns is (of coefficients of $f_i, f$ and $g$) is $\frac{dn^2}{n+1}+O(d^2)=O(d^2)$. Coefficients of the left hand side of the equations are polynomials in $\lambda$, of degrees and sizes being $O(d)$ (coming from $H^j$) and $ n^{O(d)}$ correspondingly. By assumption, on the right hand side are  polynomials of degree at most $\nu_1$, divided by some common polynomial of degree $\nu_2$, $\nu_1+\nu_2=\nu$. Their sizes are at most $s$. Applying Cramer rule, we conclude that the coefficients of $f, g$ and $f_i$ can be chosen to by polynomials of degree $\nu_1+O(d^3)$ divided by the same common denominator, so of degree $\nu+O(d^3)$, and of sizes $sd^{O(d^3)}$. Since the denominators are the same, the same bounds hold for  the degrees and sizes of $f, g$ and $f_i$.
\end{proof}

\subsection{Chen homomorphism}
Here we prove an analogue of the first claim of Theorem~\ref{th:non-deg-AI} for iterated integrals.

Let $U$ be semigroup algebra corresponding to a semigroup freely generated by formal variables $X_1, \dots , X_N$. We denote the units of the semigroup and of $\pi_1(\Gamma_H, p_0)$ by $e$. Let also $\hat U$ be a completion of $U$ in Krull topology corresponding to the maximal ideal $\mathfrak{m}=\langle X_1, \dots , X_N\rangle$.
\begin{definition}\label{def-chen}
Define the  Chen homomorphism $\varphi:\pi_1(\Gamma_H, p_0) \to U$ as
\begin{equation}\label{eq:vaphi-def}
  \varphi(\delta) = e + \sum_{(\omega_{i_1}\cdots\omega_{i_k})}\oint_\delta \omega_{i_1}\cdots\omega_{i_k} X_{i_1} \cdots X_{i_k},
\end{equation}
where summation is over the set of all non-empty words in alphabet $\omega_l$.

Let $j^K:U\to U/\mathfrak{m}^{K+1}U$ be the natural homomorphism. Define $\varphi_K$ as the composition $\varphi_K=j^K \varphi:\pi_1(\Gamma_H, p_0)\to U/\mathfrak{m}^{K+1}U$.
\end{definition}
One can easily show that $\varphi$ (and therefore $\varphi_K$) is a group homomorphism to  the set of invertible elements of $U$ (of $U/\mathfrak{m}^{K+1}$ resp.), see \cite{harris}.

Note that the space $U/\mathfrak{m}^{K+1}U$ is finite-dimensional, and has standard basis of monomials  $\{X_{i_1}X_{i_2} \cdot\dots\cdot X_{i_k}, 0\le k\le K, 1\le i_j\le N\}$. We claim that the image of $\varphi_K$ spans $ U/\mathfrak{m}^{K+1}U$.

\begin{lemma}\label{spans}
Let $\Delta^{\leq K}$ be the set of products of length at most $K$ of the generators $\delta_j$ of $\pi_1(\Gamma_H, p_0)$.
The set  $\{j^K \varphi(\delta), \delta \in \Delta^{\leq k}\}$ is a basis of $ U/\mathfrak{m}^{K+1}U$.
\end{lemma}

\begin{proof}
Claim of the proposition holds simultaneously for all base in $H^1(\Gamma_H, \Z)$, so
 we can assume that  $\{[\delta_i]\}\subset H_1(\Gamma_H, \Z)$ is a dual basis to
 $\{[\omega_i]\}\subset H^1(\Gamma_H, \Z)$.
We have $\varphi(e)=e$. This implies the statement for $K=0$.

If $K = 1$, then $\varphi(\delta_i)-\varphi(e)= X_i + \mathfrak{m}^{2}U $, so $X_i$ is in the span of $j^1 \varphi(\Delta^{\le 1})$.

For  $K>1$, we see from the previous equality that
$$
X_{i_1}X_{i_2} \cdots X_{i_k}=(\varphi(\delta_{i_1})-\varphi(e))\cdot\dots\cdot(\varphi(\delta_{i_k})-\varphi(e))+ \mathfrak{m}^{K+1}U,
$$
and, since $\varphi$ is homomorphism, the right hand side is a linear combination of elements of $\{\varphi(\delta), \delta \in \Delta^{\leq k}\}$ (mod $\mathfrak{m}^{K+1}U$).

So $\{\varphi_K(\delta), \delta \in \Delta^{\leq k}\}$ spans $ U/\mathfrak{m}^{K+1}U$, and, by cardinality reason (here we use that $\pi_1(\Gamma_H, p_0)$ is a free group), is a basis.
\end{proof}

\subsection{Construction of the horizontal section}
Abelian integrals are iterated integrals of length $1$.
The direct analogue of the matrix ${S}_1$ of \eqref{eq:S1def} for iterated integrals of length at most $K$ is the matrix
\begin{equation}\label{Sk-def}
{S}_K(H)=\left\{\oint_{\delta_{j_1}...\delta_{j_k}}\omega_{i_1}...\omega_{i_l}, \quad j_r,i_s=1,...,N\right\}_{k,l=0}^K
\end{equation}
of iterated integrals of length at most $K$ of basic forms $\omega_j\in\mathcal{B}$ over the cycles $\delta=\delta_{j_1}...\delta_{j_k}\in\Delta^{\le K}$  (we adopt convention $\int_\delta\emptyset=1$). We call these integrals \emph{the basic iterated integrals}.

 The iterated integrals depend on the choice of the base point of $\pi_1(\Gamma_H, p_0)$, so we choose $p_0$ as one of the points
of intersection of $\{H=0\}$ with the line $\sigma=\{x=0\}$ (generically, there are $n+1$ such points).
Columns of are $ {S}_K$ are  just the coordinates of  $\varphi_K(\delta), \delta\in\Delta^{\le K}$ written in the standard basis of $ U/\mathfrak{m}^{K+1}U$.

For a generic $H$ for all $\widetilde{H}$ sufficiently close to $H$ the  pairs  $\left(\Gamma_{\widetilde{H}},p_0(\widetilde{H})\right)$ are diffeomorphic to $\left(\Gamma_{H}, p_0(H)\right)$ by a diffeomorphism close to identity. This diffeomorphism is unique up to isotopy, so we can identify the fundamental groups $\pi_1\left(\Gamma_{\widetilde{H}}, p_0(H)\right)$. This means that  any path $\delta\in \pi_1(\Gamma_{H}, p_0(H))$ can be continuously extended to a family $\delta(\widetilde{H})$defined in some neighborhood of  $H$. Therefore  $ {S}_K$ can be extended holomorphically to some neighborhood of $H$, and, by analytic continuation, to a  multivalued matrix function holomorphic on some Zarisky open subset of $P\C_{n+1}[x,y]$.

Lemma~\ref{spans} claims that $ {S}_K$ non-degenerate for a generic choice of $H$.
Therefore near generic $H$ the matrix  $ {S}_K$ describes a basis of sections of  the trivial vector bundle $P\C_{n+1}[x,y]\times U/\mathfrak{m}^{K+1}$.

Our goal is to explicitly write coefficients of the connection for which the matrix $ {S}_K$ is a basis of horizontal sections.
We construct this connection locally in a neighborhood of some generic point $H\in P\C_{n+1}[x,y]$. The coefficients of the connection matrix $ {\Omega}_K=d {S}_K\, {S}_K^{-1}$ turn out to depend rationally on $H$ and the point of intersection $p\in\{H=0\}\cap\sigma$,
which is an algebraic function of $H$.

To eliminate the algebraic multivaluedness of  $ {\Omega}_K$ we lift the bundle and the connection to the corresponding algebraic cover. Namely, for any $n>0$ we define $B_n$  to be the product
\begin{equation}\label{def:Bn}
B_n=P\widetilde{\C}_{n+1}[x,y]\times\C P^1
 \end{equation}
of the space of all polynomials of degree at most $n+1$ vanishing at $(0,0)$, and of the line $\sigma=\{x=0\}$. Define mapping $\ev:B_n\to P\C_{n+1}[x,y]$ by $\ev(H,y)=H-H(y)$. Lifting $\widetilde{\Omega}_K=\ev^* {\Omega}_K$ defines a meromorphic connection on $B_n\times U/\mathfrak{m}^{k+1}$.
We prove that the resulting connections matrix is rational on $B_n$ and satisfies the conditions of Theorem~\ref{th:upper-bound}.

\subsection{Differentiation of iterated integrals}\label{sec:diff}

Our main tool in construction of the connection
is a formula of differentiating of integrals, a version of the Gelfand-Leray formula for non-closed paths. We follow closely \cite{gavrilov:iterated}.

Let $R$ be a functions holomorphic in some open set $W\subset\C^2$, and assume that its non-critical level $\{R=0\}$ is smooth and intersects transversally the line $\sigma=\{x=0\}$ at point $p_0(0)$.
Choose a path $\delta$ lying on $\{R=0\}$ and starting from $p_0(0)$  with  endpoint $p_1(0)$. For any point $p$ in a neighborhood of  $p_1(0)$ we can define a path $\delta(p)$ close to $\delta$, lying on $\{R=R(p)\}$ and joining $p$ and the point $p_0(p)$ of transversal intersection of $\{R=R(p)\}\cap\{x=0\}$.

\begin{proposition}(\cite[Lemma 2.2]{gavrilov:iterated})\label{prop:diff iter int} Let $\omega$ be a differential 1-form holomorphic in $W$. Then  for the integral $\int_{\delta(p)}\omega$, the following equation holds
\begin{equation}\label{g-l-path}
d\int_{\delta(p)}\omega = \left(\int_{\delta(p)}\frac{d\omega}{dR}\right)dR+\omega - (\sigma \circ R)^*\omega,
\end{equation}
where $\sigma: (\C,0)\to \{x=0\}$ is the parameterization of $\{x=0\}$ by values of $R$: $\sigma(t)=\{R=t\}\cap\{x=0\}$.
\end{proposition}

Assume now that the initial path $\delta=\delta(0)$ is closed, and the endpoint $p$ of the path $\delta(p)$ varies on $\sigma$. As in Definition~\ref{def:Melnikov},  denote the resulting nest of cycles by $\delta(t)$, where $t=R(p)$.
Let  $\omega_1, ..., \omega_n$ be   differential 1-forms holomorphic near $\delta_0$. Assume in addition that the pullbacks $(\sigma \circ R)^*\omega_i = 0$ for the transversal line $\sigma$.

\begin{proposition}
The following equation holds:
\begin{align}
\nonumber \frac{d}{dt}\oint  \omega_1 \dots \omega_n = \sum_{i=1}^n \oint \omega_1 \dots \omega_{i-1}\frac{d\omega_i}{dR} \omega_{i+1} \dots \omega_n \qquad\qquad\qquad\qquad\qquad\phantom{0}\\
\phantom{0}\qquad\qquad\qquad\qquad\qquad-\sum_{i=1}^{n-1}  \oint \omega_1 \dots \omega_{i-1}\frac{\omega_i \wedge \omega_{i+1}}{dR}\omega_{i+2} \dots \omega_n
\end{align}
\end{proposition}
\begin{proof} Let us denote  $\eta_i = \omega_1 \dots \omega_i$ and $\theta_j = \omega_j \dots \omega_n$.
Denote
\begin{equation*}
\begin{split}
\varphi_i(p) & = \int_{p_0}^{p} \omega_i \dots \omega_n  = \int_{p_0}^{p}\theta_i\\
\varphi_{n+1} & \equiv 1
\end{split}
\end{equation*}
Also let us define $\psi_i(q) = \int_{p_0}^q \rho_i$, where
$$
\rho_i = \frac{d(\omega_i\varphi_{i+1})}{dR}
$$
Then we have
\begin{equation*}
\begin{split}
\int_{p_0}^p \eta_{i-1}\rho_i & = \int_{p_0}^p \eta_{i-1}\frac{d(\omega_i\varphi_{i+1})}{dR} = \int_{p_0}^p \eta_{i-1}\frac{d\varphi_{i+1} \wedge \omega_i + \varphi_{i+1}d\omega_i}{dR} ~~~~ 1 \leq i < n\\
\int_{p_0}^p \eta_{n-1}\rho_n & = \int_{p_0}^p \eta_{n-1} \frac{d\omega_n}{dR}
\end{split}
\end{equation*}
and, by \ref{g-l-path}
\begin{equation}\label{phi_i}
    d\varphi_{i+1} = \psi_{i+1} dR + \omega_{i+1}\varphi_{i+2}
\end{equation}
hence
$$
    \frac{d\varphi_{i+1} \wedge \omega_i}{dR} = \omega_i\psi_{i+1} - \frac{\omega_i \wedge \omega_{i+1}}{dR}\varphi_{i+2}
$$
and then
\begin{equation*}
\begin{split}
\int_{p_0}^p \eta_{i-1}\rho_i & = \int_{p_0}^p \eta_i\psi_{i+1} - \int_{p_0}^p \eta_{i-1}\frac{\omega_i \wedge \omega_{i+1}}{dR}\theta_{i+2} + \int_{p_0}^p \eta_{i-1}\frac{d\omega_i}{dR}\theta_{i+1} ~~~~ 1 \leq i < n\\
\int_{p_0}^p \eta_{n-1}\rho_n & = \int_{p_0}^p \eta_{n-1} \frac{d\omega_n}{dR}
\end{split}
\end{equation*}
Observe that
$$
\int\eta_i\psi_{i+1} = \int\eta_i\rho_{i+1}
$$
So we obtain
$$
\int_{p_0}^p \rho_1 = \sum_{i=1}^n \int_{p_0}^p \eta_{i-1}\frac{d\omega_i}{dR}\theta_{i+1} - \sum_{i=1}^{n-1}  \int_{p_0}^p \eta_{i-1}\frac{\omega_i \wedge \omega_{i+1}}{dR}\theta_{i+2}
$$
Now assume that $p = p_0$, so $\delta(t)$ are cycles. We will use Gelfand-Leray formula to obtain
$$
\frac{d}{dt}\oint \theta_1 = \frac{d}{dt}\oint \omega_1 \varphi_2 =\oint\frac{d(\omega_1 \varphi_2)}{dR} = \oint\rho_1
$$
Hence
$$
\frac{d}{dt}\oint \theta_1 = \sum_{i=1}^n \oint \eta_{i-1}\frac{d\omega_i}{dR}\theta_{i+1} - \sum_{i=1}^{n-1}  \oint \eta_{i-1}\frac{\omega_i \wedge \omega_{i+1}}{dR}\theta_{i+2}
$$
\end{proof}
\begin{corollary}\label{prop:g-l-iterated}
Let us assume that $\omega_i = x^{\beta_i}y^{\gamma_i}dx$, then
\begin{equation}\label{g-l-iterated}
\frac{d}{dt}\oint\omega_1 \dots \omega_n = \sum_{i=1}^n\oint\omega_1 \dots \omega_{i-1}\frac{d\omega_i}{dR}\omega_{i+1} \dots \omega_n
\end{equation}
\end{corollary}
\begin{proof}
    True since $(\tau\circ R)^*\omega_i = 0$ and $\omega_i \wedge \omega_j = 0$.
\end{proof}

\subsection{Exact forms in iterated integrals}\label{sec:int-by-parts}
Let $\omega_1, \dots, \omega_n$ be a holomorphic differential 1-forms, $g$ be a holomorphic function in a domain $V$ and $\delta\subset V$ be a path connecting points $p_0$ and $p_1$. Our goal is to express iterated integrals involving the exact form $dg$ in terms of iterated integrals of smaller length.

Clearly
$$
\int_{p_0}^{p_1}dg = g(p_1) - g(p_0)
$$
For iterated integrals of length greater than $1$, integrating by parts and using \eqref{phi_i} gives
$$
\int_{p_0}^{p_1}(dg)\omega_1 \dots \omega_n = \int_{p_0}^{p_1}(dg)\varphi_1 = g\varphi_1\bigg\vert_{p_0}^{p_1} -\int_{p_0}^{p_1}g(\psi_1dR+\omega_1\varphi_2).
$$
But $dR = 0$ on level curve, so we have
$$
\int_{p_0}^{p_1}(dg)\omega_1 \dots \omega_n = g(p_1)\int_{p_0}^{p_1} \omega_1 \dots \omega_n - \int_{p_0}^{p_1}(g\omega_1)\omega_2 \dots \omega_n
$$
Next,
$$
\int_{p_0}^{p_1}\eta_i (dg) \theta_{i+1} = \int_{p_0}^{p_1} \eta_i \left(g(q)\int_{p_0}^q\theta_{i+1} - \int_{p_0}^q (g\omega_{i+1})\theta_{i+2} \right)
$$
Hence
\begin{equation}\label{eq:exform}
\int_{p_0}^{p_1}\omega_1 \dots \omega_i (dg) \omega_{i+1} \dots \omega_n = \int_{p_0}^{p_1}\omega_1 \dots (\omega_i g) \omega_{i+1} \dots \omega_n - \int_{p_0}^{p_1}\omega_1 \dots \omega_i (g \omega_{i+1}) \dots \omega_n
\end{equation}
And the third formula:
$$
\int_{p_0}^{p_1}\omega_1 \dots \omega_n(dg) = \int_{p_0}^{p_1}\omega_1 \dots \omega_n(g(q) - g(p_0)) = \int_{p_0}^{p_1}\omega_1 \dots (\omega_ng) - g(p_0)\int_{p_0}^{p_1}\omega_1 \dots \omega_n
$$

\subsection{Construction of Picard-Fuchs system}
Let $H$ be a Hamiltonian of degree $n+1$, which we can write in multi-index form
$$
    H = \sum_{0\le i+j\leq n+1} \lambda_{ij} x^iy^j, \quad \lambda=(\lambda_{ij})\in P\C_{n+1}[x,y].
$$
We assume that $\mathcal{B} =\{x^{i-1}y^jdx\}= \{\omega_l, l=1,...,n^2\}$ form  a basis of the Petrov module $P_H$, and that the curve $\{H=0\}$ is smooth and intersects the line $\sigma=\{x=0\}$ transversally. We will compute the connection matrix $\Omega_K$ locally near $H$, and then, by analytic continuation, this expression will be valid everywhere.

Let $\delta \subset \{H=0\}$ be a cycle with an initial point at $p(H)\in\delta\cap\sigma$, and consider the vector of coefficients of $\varphi(\delta)$ of \eqref{eq:vaphi-def} in the standard basis $\{X_{i_1}\cdot\dots\cdot X_{i_k}, k=1,...,K\}$ of $U/\mathfrak{m}^{K+1}$:
$$
I(\lambda)=
\begin{pmatrix}
1\\I_1\\I_2 \\...\\I_{N_K}
\end{pmatrix}
$$
where $\lambda = \{\lambda_\alpha\}_{|\alpha|\leq n+1}$. We assume that the integrals $I_1, \dots, I_{N_K}$ are ordered by length, i.e. $I_j = \int \eta_1 \dots \eta_k$ if and only if $N_{k-1}<j\leq N_k$, where $N_k=\dim U/\mathfrak{m}^{k+1}$.

Our first goal is to provide an analogue of Proposition~\ref{prop:petrov} for iterated integrals.

\begin{proposition}\label{lin-eq-1}
 If $\eta_1\in\mathcal{B}, \dots , \eta_K\in\mathcal{B}$, and $\theta$ is a differential 1-form of degree $d$, then, for any $1\leq i \leq K+1$,
\begin{equation}\label{bigsum}
    \oint_\delta \eta_1 \dots \eta_{i-1}\theta\eta_i \dots \eta_K = \sum_{j=1}^{N_{K+1}}h_j I_j, \qquad h_j\in\C(\lambda)[p],
\end{equation}
where $p=\delta(0)$ is the initial point of the cycle $\delta$. Degrees of $h_j$ in $p$ do not exceed $d+O(nK)$.

Moreover, if $\theta$ is defined over $\Q(\lambda)$, its degree and size do not exceed $\nu$ and $s$ correspondingly and $d\ge n$, then the
the degrees and  sizes of $h_j$ do not exceed  $\nu+O(K^5d^4)$ and $s(Kd)^{O(K^5d^3)}$ respectively.
\end{proposition}
\begin{proof}
Using Proposition~\ref{prop:petrov}, we can write
\begin{equation}\label{subst-to-long}
\begin{split}
    \oint\eta_1 \dots \eta_{i-1}\theta\eta_i \dots \eta_K = \oint\eta_1 \dots \eta_{i-1}\left(\sum_{j=1}^m (f_j\circ H)\omega_j+dg\right)\eta_i \dots \eta_K\\
    = \sum_{j=1}^m f_j(0)\oint\eta_1 \dots \eta_{i-1}\omega_j\eta_i \dots \eta_K+\oint\eta_1 \dots \eta_{i-1}(dg)\eta_i \dots \eta_K,
\end{split}
\end{equation}
since $\delta\subset\{H=0\}$. The latter term can be rewritten as a sum of two iterated integrals of length $K$, using the equations of  \S\ref{sec:int-by-parts}, or, if $i=1$ or $i=K+1$, as a sum of an iterated integral of length $K$ and $g(p)\oint_\delta \eta_1 \dots  \eta_K$, a multiple of a basic integral. This proves (by induction) the formula \eqref{bigsum}.

Bounds on the size and degree follow from this inductive proof and bounds of Proposition~\ref{prop:petrov}.
Note that iterated integrals obtained on different steps  of the induction have different lengths.
First, one can see that $\deg_{p}h_j\le  d + \sum \deg\eta_i\le d+2nK$, as dependence on $p$ appears only in the remainder $g(p)$ above and their degrees are bounded by $d+2nK$.

One can show that all  denominators of the polynomials $f_i, g$ obtained on al inductive steps are factors of
$\Delta=\left(\prod_{j=1}^{d+2nK} \Delta_j\right)^K$, where $\Delta_j$ is the determinant of the system \eqref{eq:petrov decomp} corresponding to a form $\theta$ od degree $j$. We get $\deg\Delta=O(K^5d^4)$ by Proposition~\ref{prop:petrov}.
Therefore, multiplying $\theta$ by $\Delta$, we can assume that all $f_i, g$ on all steps of induction are polynomial in $\lambda$.
Denote by $b(K+1, d, \nu)$ the degrees of their coefficients. Then
$$
b(K+1, d, \nu)\le \max\{\nu+O(d^3),b(K, d+2n, \nu+O(d^3))\}
$$
(polynomiality allows to replace sums of degrees in \eqref{eq:deg sum rat} by maximum of the degrees). This implies $b(K+1, d, \nu)\le \nu+O(K^4d^3)$. Adding the degree of $\Delta$, we get the required estimate.

Similarly, assume that the polynomials $f_i, g$ obtained on all inductive steps are polynomials, and denote by $\mathfrak{s}(K+1, d, s)$ the size of the polynomials $h_j$.
We have
$$
\mathfrak{s}(K+1, d, s)\le \max\left\{sd^{O(d^3)},2\mathfrak{s}(K, d+2n, sd^{O(d^3)})\right\},
$$
where factor $2$ appears due to \eqref{eq:exform}.
This implies that
$$
\mathfrak{s}(K+1, d, s)\le s(Kd)^{O(K^4d^3)}.
$$
Now, size of $\Delta$ is $(Kn)^{O(K^5n^3)}$, so applying the previous estimate to $\Delta\theta$ (and remembering $d>n$), we get the result.
\end{proof}

We can prove more general statement:
\begin{proposition}\label{prop:iter Petrov}
Let $\theta_1, \dots, \theta_K\in \Lambda^1(\C^2)\otimes\C(\lambda)$ be  1-forms of degree at most $d$, and of degree in $\lambda$ at most $\nu$.
Then
\begin{equation}\label{eq:iter Petrov}
    \oint_\delta \theta_1 \dots \theta_K = \sum_{j=1}^{N_K}h_j I_j,\qquad h_j\in\C(\lambda)[p],
\end{equation}
with  degrees of $h_j$ in $\lambda,p$  bounded by  $ \nu d^{O(K^2)}$.

\end{proposition}
\begin{proof}
Indeed, decomposing $\theta_i=\sum_{\omega\in\mathcal{B}} f_{i\omega}\circ H\omega + f_idH+dg_i$ as in Proposition~\ref{prop:petrov}, we see that
\begin{equation}
\oint_\delta\theta_1\dots\theta_K=\sum_{\phi}  \left(\prod_{i=1}^Kf_{i\phi(i)}\right)\oint_\delta{\phi(1)}...{\phi(K)}+...,
\end{equation}
where summation is over all mappings $\phi:\{1,...,K\}\to\mathcal{B}$ and the dots denote $(n^2+1)^K-n^{2K}$ iterated integrals of length $K$ with at least one exact form $dg_i$. These can be represented as iterated integrals of lesser length, and the result follow by induction.

To estimate the degrees and sizes of the coefficients in \eqref{bigsum}, note that the degrees in $\lambda$ of $ \prod_{i=1}^Kf_{i\phi(i)}$ are $K\nu +{O(Kd^3)}$ by Proposition~\ref{prop:petrov}. The remaining $(n^2+1)^K-n^{2K}$ terms can be rewritten by formulae of \S\ref{sec:int-by-parts}, as  at most $2\left((n^2+1)^K-n^{2K}\right)$ iterated integrals of length $K-1$ with coefficients being rational in $\lambda$ and polynomial in $p$, of degree at most $K\nu + {O(Kd^3)}$ in $\lambda$ and at most $d$ in $p$. Under the integral sign stand some tuple of basic forms, $dg_i$-s  and product of some $g_i$ with one of them. So  these are forms of degrees at most $3d$. Since the degree of the sum of rational functions is at most double their total sum, we get
$$
b(K,d,\nu)\le  4\left(K\nu +{O(Kd^3)}+b\left(K-1,3d,2\nu+O(d^3)\right)\right)\left((n^2+1)^K-n^{2K}\right),
$$
where $b(K,d,\nu)$ denote an upper bound for the degree. This implies that $b(K,d,\nu)\le \nu d^{O(K^2)}$.
\end{proof}

However, in order to prove Theorem~\ref{th:upper-bound}, we will need only dependence on $\lambda_{00}=-t$ and on $p$, which are much better:
\begin{proposition}\label{prop:iter Pertov in p}
Assume that the forms $\theta_i$ are independent on $\lambda_{00}$. Then the  coefficients  $h_j$ in \eqref{eq:iter Petrov} are polynomial in $p, \lambda_{00}$, and for their degrees in $p, \lambda_{00}$ we get
\begin{equation}\label{eq:iter Petrov in p}
\deg H\deg_{\lambda_{00}}h_j+\deg_ph_j+\sum\deg\omega_{j_k}=\sum\deg\theta_i.
\end{equation}
\end{proposition}
\begin{proof}
Polynomiality follows from Proposition~\ref{prop:petrov}. Therefore the bound can be computed by asymptotics at infinity, as $t\to\infty$.
For homogeneous forms and generic homogeneous Hamiltonians counting homogeneity degrees we get
$$
\deg H\deg_{\lambda_{00}}h_j+\deg_ph_j+\sum\deg\omega_{j_k}=\sum\deg\theta_i,
$$
where $I_j=\int\omega_{j_1}...\omega_{j_m}$. In general case, this becomes an inequality.
\end{proof}

\begin{proposition}\label{prop:tildeOmega K}
Let, as before, $H$ be a Morse-plus polynomial such that the curve $\Gamma_H=\{H=0\}$ is smooth and intersects transversally the line $\sigma=\{x=0\}$. Let $ {S}_K(\lambda)$ be the  $N_K \times N_K$-matrix valued function defined in a neighborhood of $H$ as in \eqref{Sk-def}.
Then
\begin{equation}\label{g-l-long-matrix}
d {S}_K(\lambda) =  {\Omega}_K(\lambda,p) {S}_K(\lambda),
\end{equation}
with  $N_K \times N_K$ matrix of $ {\Omega}_K(\lambda,p(\lambda))$ of one-forms on $P\C_{n+1}[x,y]$ with coefficients being rational functions of $\lambda$ and polynomial in $p(\lambda)$. Here $p=p_0(\lambda)$ is a starting point of integration, $p\in\Gamma_H\cap\sigma$.

The coefficients of $ {\Omega}_K(\lambda,p)$ have degree in $p$ at most $O(nK)$, and their degrees in $\lambda$ and sizes are at most $O(K^6n^7)$ and $(Kn)^{O(K^6n^5)}$.
\end{proposition}
\begin{proof} According to Proposition~\ref{prop:g-l-iterated} for any tuple $\{\eta_i\}_{i=1}^K$, with $\eta_i\in\mathcal{B}$, 
\begin{equation}\label{g-l-rat}
    \frac{\partial}{\partial\lambda_\alpha}\oint \eta_1 \dots \eta_K = -\sum_{i=1}^K \oint \eta_1 \dots \eta_{i-1}\left(\frac {d\eta_i}{d\lambda_\alpha}\right) \eta_{i+1} \dots \eta_K.
\end{equation}
Our goal is to express the Gelfand-Leray derivatives $\frac {d\omega_i}{d\lambda_\alpha}$ as rational combinations of forms $\omega_j$.
Consider the form $H d\omega_i$. Differentials of the elements of  $\mathcal{B}$ form a basis of $\Lambda^2(\C^2)/dH\wedge\Lambda^1(\C^2)$, so one can write
\begin{equation}\label{eq:JacId}
H d\omega_i=\sum_{j=1}^{n^2} a_{ij}d\omega_j+\eta_i\wedge dH,\quad i=1,...,n^2.
\end{equation}
Multiplying by the monomial $\frac{\partial H}{\partial \lambda_\alpha}$, dividing by $dH$ and decomposing $\frac{\partial H}{\partial \lambda_\alpha}\eta_i$ as in  Proposition~\ref{prop:petrov}
we get, modulo some multiple of $dH$,
\begin{equation}\label{eq:GMconn}
H\frac{d\omega_i}{d \lambda_\alpha}=\sum_j a_{ij}\frac{d\omega_j}{d\lambda_\alpha}+\left(\sum_j b^\alpha_{ij}(H) \omega_j+f^\alpha_idH+d\tilde{g}^\alpha_i\right), \qquad i,j=1,...,n^2.
\end{equation}
(we follow closely \cite[A.3]{BNY-Inf16}). Multiplying both sides of this system of equations by inverse $\breve{A}$ of the matrix $A=\{a_{ij}\}$ we get
\begin{equation}\label{rat-coef}
    H\sum \breve{a}_{ij}\frac{d\omega_j}{d \lambda_\alpha}=\frac{d\omega_i}{d\lambda_\alpha} + \sum_{j=1}^{n^2} q_{ij}^\alpha(H)\omega_j+ dg^\alpha_i
\end{equation}
where coefficients $q^\alpha_{ij}(H)$ are polynomial in $H$ with coefficients being rational functions of $\lambda$.
\par Integrating over $\delta\subset\{H=0\}$, we get
\begin{equation}\label{g-l-long-rat}
\begin{split}
    \frac{\partial}{\partial\lambda_\alpha}\oint \eta_1 \dots \eta_K = -\sum_{i=1}^K \sum_{j=1}^{n^2} q_{ij}^\alpha(0)\oint \eta_1 \dots \eta_{i-1} (\omega_j) \eta_{i+1} \dots \eta_n \\
    - \sum_{i=1}^K  \oint \eta_1 \dots \eta_{i-1} (dg_i^\alpha) \eta_{i+1} \dots \eta_n
\end{split}
\end{equation}
Using Proposition~\ref{lin-eq-1}, we can express the integrals from the right hand side of the latter equation as a combination of basic ones.

Now, let us estimate the degrees and sizes of the coefficients of  $ {\Omega}_K(\lambda,p)$.
The degree in $x,y$ of the form $\eta_i$ is equal to the degree of $\omega_i$, i.e. is at most $2n$. To find the degrees in $\lambda$ and sizes in \eqref{eq:JacId} note that this is a system of $O(n^2)$ linear equations on $a_{ij}$ and coefficients of $\eta_i$ of degree $1$ in $\lambda$ and of size $O(n)$. Therefore $a_{ij}$ and $\eta_i$ can be chosen of degree $O(n^2)$ in $\lambda$ and of size $n^{O(n^2)}$.

Bounds of Proposition~\ref{prop:petrov} imply that the degrees $\deg_H b^\alpha_{ij}\le 3=O(1)$, degree in $x,y$ of $\tilde{g}^\alpha_i$ is bounded by $O(n)$, and degrees in $\lambda$ and sizes of $b^\alpha_{ij}$, $\tilde{g}^\alpha_i$ are bounded above by $O(n^3)$ and $n^{O(n^3)}$ respectively.  This implies that degrees and sizes of $\breve{a}_{ij}$ are bounded by $O(n^4)$ and $n^{O(n^2)}$ respectively, so $q_{ij}^\alpha$ are polynomials in $H$ of degree $O(1)$, of degree $O(n^6)$ in $\lambda$ and of size $n^{O(n^2)}$. The functions $g^\alpha_i$ are polynomial in $x,y$ of degree $O(n)$, and their degrees in $\lambda$ and sizes are $O(n^7)$ and $n^{O(n^5)}$  correspondingly. Bounds of Proposition~\ref{lin-eq-1} now imply that the coefficients of $ {\Omega}_K(\lambda,p)$ are rational functions of $p,\lambda$, polynomial in $p$ of degree at most $O(nK)$, and of degrees in $\lambda$ and sizes at most $O(K^6n^7)$ and $(Kn)^{O(K^6n^5)}$.
\end{proof}
\subsection{Changing the variables (lifting)}
The coefficients of the connection \eqref{g-l-long-matrix} depend algebraically on $\lambda\in P\C_{n+1}[x,y]$, since the base point $p=p(\lambda)$ of the fundamental group
is not defined uniquely. Let $\ev : B_n\to P\C_{n+1}[x,y]$ be the map $\ev(H,y)=H-H(y)$, where $B_n$ was defined in \eqref{def:Bn}. Let $\widetilde{S}_K=S_K\circ\ev$ be the lifting of the matrix $S_K$ to $B_n$. The coefficients of the pulled-back connection
 \begin{equation}\label{eq:Omega_K}
d\widetilde{S}_K=\widetilde{\Omega}_K \widetilde{S}_K, \quad \Omega_K=\ev^* {\Omega}_K,
\end{equation}
on $B_n\times U/\mathfrak{m}^{K+1}$ are rational one-forms on $B_n$.

\begin{proposition}\label{prop:II sys compl}
Degree and the size of the matrix $\widetilde{\Omega}_K$ are bounded as  $O(K^6n^8)$ and $(Kn)^{O(K^6n^7)}$ correspondingly.
\end{proposition}
\begin{proof}
Degree of the mapping $\ev$ is $n+1$, and it has  coefficients equal to $0$ or $1$.  Therefore degree of  $\widetilde{\Omega}_K$ is at most $O(K^6n^8)$. Sizes of coefficients of $\widetilde{\Omega}_K$  will not exceed sizes of coefficients of $ {\Omega}_K(\lambda,p)$ multiplied by the size of $\ev$ raised to $\deg_{\lambda} {\Omega}_K(\lambda,p)$, i.e. $(Kn)^{O(K^6n^5)} n^{O(K^6n^7)}=(Kn)^{O(K^6n^7)}$.
\end{proof}

\section{Properties of the system}\label{sec:properties}

\subsection{Quasi-Unipotency}

\begin{proposition}\label{prop:quasi-unipot regular}
   Connection \eqref{eq:Omega_K} is  quasi-unipotent and regular.
\end{proposition}
 \begin{proof}
From  \eqref{g-l-long-rat} we see that derivatives of an iterated integral are linear combination of iterated integrals of smaller or equal length.
This means that $\Omega_K$ is a lower-block-triagonal matrix:
$$
\Omega_K =
\begin{pmatrix}
0&0&0& &0\\
*&\Theta_{11} & 0 &  & 0\\
  &*& \Theta_{22} & & 0\\
  & *&  & \cdots & 0\\
  & * & & & \Theta_{nn}
\end{pmatrix},
$$
where each block $\Theta_{ii}$ is $k_i\times k_i$ matrix corresponding to the integrals of length exactly $i$
Note that $\Theta_{11}$ is just the pull-back by $\ev^*$ of the matrix of the Gauss-Manin connection for Abelian integrals, so has quasiunipotent monodromy by \cite{kashiwara:qu}.

Since $\widetilde{\Omega}_K$ is lower-block-triagonal, it preserves the flag $\mathcal{F}=\{0\}=\mathcal{F}_{K+1}\subset \mathcal{F}_K\subset...\subset \mathcal{F}_0=U/\mathfrak{m}^{K+1}$, where $\mathcal{F}_i= \mathfrak{m}^{i} /\mathfrak{m}^{K+1}$,
so any monodromy operator corresponding to $\Omega_K$ preserves this flag as well, i.e. is lower-triangular with blocks $M_{ii}$ on diagonal. Therefore its eigenvalues are just the eigenvalues of the monodromy operators $M_{ii}$ corresponding to $\Theta_{ii}$, the connection induced by $\widetilde{\Omega}_K$ on the factor-bundle with fiber $\mathcal{F}_i/\mathcal{F}_{i+1}=\mathfrak{m}^{i} /\mathfrak{m}^{i+1}$.

\begin{lemma}\label{kroneker}
For all $1 \leq k \leq n$, $\Theta_{kk}= P(\bigoplus_{i=1}^k \Theta_{11})P^{-1}$ for some permutation matrix $P$, where we use the notation of \emph{Kroneker sum}: $A\oplus A = A \otimes I + I \otimes A$. Correspondingly, $M_{kk}= PM_{11}^{\otimes k}P^{-1}$.
\end{lemma}

The first claim is just a way to say that the differentiation of iterated integrals satisfies Leibnitz rule, up to iterated integrals of lesser length.
This can be seen from \eqref{g-l-rat}. Now, derivation of the products of Abelian integrals satisfy the same rule, so horizontal sections of $\Theta_{kk}$ are described by these products, and monodromy operators of  are just tensor powers of monodromy $M_{11}$ of Abelian integrals.
This proves quasiunipotency of  $\widetilde{\Omega}_K$  since tensor powers of  quasiunipotent operators are quasiunipotent.

Regularity of $\widetilde{\Omega}_K$ follows from regularity of $\Theta_{kk}$ and the fact that semidirect product of regular connections is regular, see \cite{deligne}.

\end{proof}

\section{Proof of Theorem~\ref{num-of-zeroes-m-k}}\label{sec:Mk as II}

Let $\delta\subset \{H=0\}$ be a cycle and $U$ its small neighborhood, and let $M_K$ be  the first non-zero Melnikov function defined as in \eqref{eq:M_k def}. Recall the construction expressing $M_K$ as a polynomial in iterated integrals of the perturbation form $\omega$ and its Gelfand-Leray derivatives up to order $K$.

\begin{definition}\label{rel-exact}
A real analytic 1-form $\alpha\in \Lambda^1(U)$ is \emph{relatively exact} with respect to the integrable foliation $\mathcal{F} = \{dH = 0\}$ in a domain $U$, if
\begin{equation}
    \alpha = h\cdot dH + dg, ~~~ h,g \in \mathcal{O}(U)
\end{equation}
\end{definition}
Clearly, the integral of a relatively exact form $\alpha$ along any closed oval on any level curve $\{H=z\}\subset U$, vanishes:
\begin{equation}
\forall \text{ oval } \delta \subseteq \{f=z\} ~~~ \oint_\delta \alpha = 0
\end{equation}

Define the sequence of real analytic 1-forms $\omega_1, \omega_2, \dots, \omega_k$ as follows:
\begin{enumerate}
\item (Base of induction). $\omega_1 = \omega$ is the perturbation form from \eqref{perturbed}
\item (Induction step). If the forms $\omega_1, \dots, \omega_j$ are already constructed and turned out to be relatively exact, then $\omega_j = h_j\cdot dH+dg_j$. In this case we define
\begin{equation}\label{omega-induction}
\omega_{j+1} = -h_j\omega
\end{equation}
\end{enumerate}

\begin{theorem}\cite[Theorem 26.7]{iy:lade} If $\omega_k$, $k\geq 2$, is the first not relatively exact 1-form in the sequence $\omega_1, \dots, \omega_{k-1},\omega_k$, constructed inductively by \eqref{omega-induction}, then
\begin{equation}
    M_k(z) = -\oint_{\{H=z\}} \omega_k
\end{equation}
\end{theorem}

Evidently, the functions $h_j$ can be restored as $h_j=-\int\frac{d\omega_j}{dH}$, so
$$
\omega_{j+1} =\omega\int\frac{d}{dH}\left(-\omega\int\frac{d}{dH}\left(\dots\left(-\omega\int\frac{d\omega}{dH}\right)\dots\right)\right).
$$

Denote by $\phi$ the algebraic function $H|_{\sigma}^{-1}\circ H$ of $x,y$ which maps the point $(x,y)$ to the (one of $d+1$) point of the transversal $\sigma=\{x=0\}$ lying on the same level curve of $H$ as $(x,y)$. In other words, $p=\phi(x,y,\tilde{\lambda})$.
\begin{lemma}
$h_j$ in \eqref{omega-induction} is a linear combination of iterated integrals of differential one-forms with coefficients polynomial in $x,y$. The coefficients of this combination are rational in $\tilde{\lambda},p$.
\end{lemma}

\begin{proof}
We have $h_0=1$, which is of the required form trivially.
We proceed by induction on $j$. Assume that  $h_j$ is a finite sum of terms of the type  $R(\lambda,p)\int \theta_1...\theta_k$, where $\theta_i$ are polynomial in $x,y$ one-forms, and $R$ is a rational function. Now,  $h_{j+1}=\int \frac {d \left(h_j\omega\right)}{dH}$, and it is look what happens with  just one term of this sum.
Differentiation, according to \ref{prop:diff iter int}, gives
\begin{equation}\label{eq:h_j+1 is also good}
\begin{split}
\int \frac{d}{dH}\left(R(\lambda,p)\omega\int \theta_1...\theta_k\right)=\frac{\partial R}{\partial p}\left(H|_{\sigma}'(p)\right)^{-1}\int\omega \theta_1...\theta_k+\\
R\int \frac{d\omega}{dH} \theta_1...\theta_k +
R\sum_{i=1}^k\int\omega\theta_1...\frac{d\theta_i}{dH}...\theta_k+\\
R\int\omega\sum_{i=1}^{k-1} \theta_1...\frac{\theta_i\wedge\theta_{i+1}}{dH}...\theta_k+
R\int\omega \theta_1...\frac{\theta_{k-1}\wedge (\sigma\circ H)*\theta_k}{dH}.
\end{split}
\end{equation}
The first term is clearly of the required type.

Denote by $m(H)$ the product $\prod_{i=1}^{n^2}\left(H-c_i\right)$, where $c_i$ are critical values of $H$ (repeated if multiple).
It is well known, see \cite{gavrilov:iterated}[Prop.2.4], that $m(H)$ lies in the Jacobian ideal of $H$, so the operator $\frac{m(H)}{dH}$ preserves polynomial one-forms. Therefore
\begin{equation}
\begin{split}
R\int\omega\sum_{i=1}^{k-1} \theta_1...\frac{\theta_i\wedge\theta_{i+1}}{dH}=\frac R {m(H(p))} \int\omega\sum_{i=1}^{k-1} \theta_1...\frac{m(H)\theta_i\wedge\theta_{i+1}}{dH}\\
R\sum_{i=1}^k\int\omega\theta_1...\frac{d\theta_i}{dH}...\theta_k= \frac R {m(H(p))}\sum_{i=1}^k\int\omega\theta_1...\frac{m(H)d\theta_i}{dH}...\theta_k
\end{split}
\end{equation}
all terms in \eqref{eq:h_j+1 is also good} except the last one are of the required type.

Finally, $(\sigma\circ H)*\theta=\phi_k(p) dH$, where $\phi_k(p)=\frac{\theta(p)(\partial_y)}{dH(p)(\partial_y)}$ is a rational function of $p$, so the last term of \eqref{eq:h_j+1 is also good} is also of the required type.
\end{proof}

\begin{lemma}\label{lem:Mk as lincomb} For a form $\omega$ of degree $d>n$ in $x,y$, we have
$$
M_K=\sum_{i=1}^{N_K}h_iI_i,
$$
where $h_j$ depend rationally on  $p$ and has degree at most $2^{O(K)}dn^6 $ in $p$.
\end{lemma}
\begin{proof}
In the inductive step \eqref{eq:h_j+1 is also good} one iterated integral of forms of degrees at most $d_k$ with coefficient $R$ of degree $\nu_k$ generated $O(k)$ iterated integrals of forms of degrees at most $2d_k+O(n^3)$. 
The coefficients of these new integrals are obtained from $R$ by a combination of a differentiation and division either by $H|_{\sigma}'(p)$ or by $m(H(p))$ or just by division by one of these polynomials. Applying these operation $K$ times we can increase degree of $R$ by at most $O(Kn^3)$. 
Summing together, we get a representation of $M_K$ as a sum of $2^{O(K^2)}$ iterated integrals of forms of degree at most $2^{O(K)}\left(d+O(n^3)\right)$, with rational in $ p$ coefficients of degree at most $ O(Kn^3)$, with common denominator of degree $O(Kn^3)$.

Applying Proposition~\ref{prop:iter Pertov in p}, we represent each of these iterated integrals as in \eqref{eq:iter Petrov}, and the coefficients of these representations have degrees in $p$  at most $2^{O(K)}(d+n^3)$. Summing these representations together, we arrive to the statement of the Proposition.\end{proof}

Multiplying by the common denominator, we see that the estimate of Theorem~\ref{num-of-zeroes-m-k} follow from the following
\begin{lemma}\label{lem:poly env}
Linear combination $\sum_{i=1}^{N_K}h_iI_i$ of iterated integrals of length at most $K$, with coefficients $h_i$ being polynomial in $p$ of degree at most $\mu$, has at most $\exp\left(\exp\left(\mu^{O(1)} n^{O(K)}\right)\right) $ zeros on each line $\tilde{\lambda}\times \C_p$.
\end{lemma}
\begin{proof}
We can construct in a standard way a connection  whose horizontal sections are described by these functions. More exact, the connection will be a Kronecker sum of ${\widetilde{\Omega}_K}$ with the connection whose sections are polynomials in $p$ of degree at most $\mu$. Therefore it will be of dimension $\mu N_K=\mu n^{O(K)}$, of degree $O(K^6n^8)$ and of size $\mu (Kn)^{O(K^6n^7)}$. Theorem~\ref{th:upper-bound} then gives the required upper bound.
\end{proof}

Substitution of $\mu=2^{O(K)}dn^6 $ from Lemma~\ref{lem:Mk as lincomb} provides the estimate of the Theorem~\ref{num-of-zeroes-m-k}.

\bibliographystyle{amsalpha}
\def\MR{}
%\bibliography{tangent16,2003,2007}
\def\BbbR{$\mathbb R$}\def\BbbC{$\mathbb
  C$}\providecommand\cprime{$'$}\providecommand\mhy{--}\font\cyr=wncyr8\def\Bb%
bR{$\mathbb R$}\def\BbbC{$\mathbb
  C$}\providecommand\cprime{$'$}\providecommand\mhy{--}\font\cyr=wncyr9\def\Bb%
bR{$\mathbb R$}\def\BbbC{$\mathbb
  C$}\providecommand\cprime{$'$}\providecommand\mhy{--}\font\cyr=wncyr9\def\cp%
rime{$'$} \providecommand{\bysame}{\leavevmode\hbox
to3em{\hrulefill}\thinspace}
\providecommand{\MR}{\relax\ifhmode\unskip\space\fi MR }
\providecommand{\MRhref}[2]{%
  \href{http://www.ams.org/mathscinet-getitem?mr=#1}{#2}
} \providecommand{\href}[2]{#2}

\end{document}